\documentclass{article}
\usepackage{amsmath,amssymb,amsthm,graphicx,epsfig,float,url}
\usepackage[colorlinks=true,citecolor={ForestGreen},linkcolor={Periwinkle}]{hyperref}
\usepackage{pdfsync}
\usepackage[usenames, dvipsnames]{xcolor}
\usepackage{tikz}
\usepackage{subfig}
\usepackage{bbm}
\usepackage{stmaryrd}
\usepackage{amssymb}
\usepackage{mathrsfs}  
\usepackage{caption}
\usepackage{float}
\usepackage{esint}

\usepackage{dsfont}
\usepackage[makeroom]{cancel}
\usetikzlibrary{patterns}
\newcommand{\R}{\textnormal{I\kern-0.21emR}}
\newcommand{\N}{\textnormal{I\kern-0.21emN}}

\newcommand{\T}{\mathbb{T}}
\newcommand{\uu}{{u_{y}}}

\newcommand{\duy}{{\dot{u}_{y}}}

\newcommand{\dKuu}{{\dot{v}_{K,y^*}}}

\newcommand\e{{\varepsilon}}

\usepackage[dvipsnames]{xcolor}

\newtheorem*{theorem*}{Theorem}

\newtheorem{theorem}{Theorem}

\newtheorem{material}{material}
\newtheorem{proposition}[material]{Proposition}
\newtheorem{corollary}[material]{Corollary}

\newtheorem{lemma}[material]{Lemma}

\newtheorem{remark}[material]{Remark}

\def\O{{\Omega}}

\def\n{{\nabla}}

\def\p{{\varphi}}

\def\OT{{(0;T)\times \O}}

\usepackage{xargs}
 \usepackage[colorinlistoftodos,textsize=small]{todonotes}
 \newcommandx{\christian}[2][1=]{\todo[linecolor=red,backgroundcolor=red!25,bordercolor=red,#1]{#2}}
 \newcommandx{\laura}[2][1=]{\todo[linecolor=blue,backgroundcolor=blue!25,bordercolor=blue,#1]{#2}}
 \newcommandx{\info}[2][1=]{\todo[linecolor=green,backgroundcolor=green!25,bordercolor=green,#1]{#2}}
 \newcommandx{\improvement}[2][1=]{\todo[linecolor=yellow,backgroundcolor=yellow!25,bordercolor=yellow,#1]{#2}}
 
  \newcommandx{\biblio}[2][1=]{\todo[linecolor=blue,backgroundcolor=magenta!25,bordercolor=blue,#1]{#2}}

 \numberwithin{equation}{section}

\begin{document}

\title{Localising optimality conditions for the linear optimal control of semilinear equations  \emph{via}  concentration results for oscillating solutions of linear parabolic equations}


\author{Idriss Mazari-Fouquer, Gr\'egoire Nadin}

\date{ \today}

\maketitle
\begin{abstract}
We propose a fine analysis of second order optimality conditions for the optimal control of semi-linear parabolic equations with respect to the initial condition. More precisely, we investigate the following problem: maximise with respect to $y\in L^\infty(\OT)$ the cost functional $J(y)=\iint_\OT j_1(t,x,u)+\int_\O j_2(x,u(T,\cdot))$ where $\partial_t u-\Delta u=f(t,x,u)+y\,, u(0,\cdot)=u_0$ with some classical boundary conditions, under  constraints of the form $-\kappa_0\leq  y\leq \kappa_1\text{ a.e.}\,, \int_\O y(t,\cdot)=V_0$. This class of problems arises in several application fields. A challenging feature of these problems is the study of the so-called abnormal set $ \{-\kappa_0<y^*<\kappa_1\}$ where $y^*$ is an optimiser. This set is in general non-empty and it is important (for instance for numerical applications) to understand the behaviour of $y^*$ in this set: which values can $ y^*$ take? In this paper, we introduce a Laplace-type method to provide some answers to this question. This Laplace type method is of independent interest.
\end{abstract}


\medskip


\paragraph{Acknowledgement:} The authors were partially supported by the Project ”Analysis and simulation of optimal shapes - application to life sciences” of the Paris City Hall. I. M-F was partially supported by  the French ANR Project ANR-18- CE40-0013 - SHAPO on Shape Optimization.

\section{Introduction and main result}

\subsection{Scope and objective of the article}
An ubiquitous query in PDE constrained optimisation is the \emph{optimisation of a linear source term in parabolic models}. While several works  \cite{CasasT, RaymondT,RaymondZidani, RT} tackle the delicate issue of analysing second (and first) order optimality conditions under a wide class of constraints and penalisations, these works often fail to offer conclusive information in the context of $L^\infty-L^1$ constrained control problems. These type of constraints arise naturally in the context of population dynamics \cite{Lou2008}, and the recent activity in the analysis of these optimal control problems, whether it be in the elliptic \cite{MNP2021} or in the parabolic setting \cite{Abdul_Halim_2022,Elie}, has underlined the intrinsic mathematical challenges of these queries. While previous works are discussed in section \ref{Se:Prior} let us mention here that, in the present paper, we consider a general optimisation problem for heterogeneous semi-linear parabolic equations.  

The main difficulty in this endeavour is that the optimality conditions typically involve the use of an adjoint state, defined as the solution of a (backward) parabolic equation on the entire space-time domain. However, it is often desirable to obtain a pointwise  (in time and in space) information, so that \emph{localising} the optimality conditions is a worthy but intricate endeavour. The method we propose here leads to such a   \emph{localisation} of these optimality conditions, and provides unexpected results.

Besides being relevant  for  the numerical approximation of such optimal control problems \cite{Mazari_2021}, our results shed a new light on the qualitative properties of solutions of linear optimal control of semi-linear models. Furthermore, in exploiting these optimality conditions, we develop a \emph{Laplace-type method} that deals with the limit behaviour of solutions to linear parabolic equations when the initial condition is a sum  of highly oscillating frequencies, and yields a concentration-type result. This contribution is related to \emph{two-scale asymptotic expansions}. What is notable here is that we prove a result that does not assume a scale separation, unlike what is usually done in this context \cite{AllaireBriane}.

\subsection{State equation}
Throughout the paper, $\O\subset \R^d$ is a bounded open set with a $\mathscr C^2$ boundary. We choose a boundary condition operator $\mathcal B$ that is of the following form:
\[\mathcal B:u\mapsto \begin{cases} u&\text{(Dirichlet case)}
\\ \frac{\partial u}{\partial \nu}&\text{(Neumann case)}.\end{cases}\]
We work with a non-linearity $f=f(t,x,u)$ that satisfie
\begin{equation}\tag{$\bold{H}_f$}\label{Hyp:f}
\begin{cases}
\text{ $f$ is $\mathscr C^3$ on $[0;T]\times \overline \O\times \R$}
\\
\exists M>0\,, \forall u\geq M\,, \forall (t,x)\in [0;T]\times \O\,, f(t,x,u)\leq 0\,, f(t,x,-u)\geq 0.\\
\end{cases}
\end{equation}
 For any initial condition $u_0\in L^\infty(\O)$ and any source term $y\in L^1(0,T;L^1(\O))\cap L^\infty(\OT)$, we define $u_{y}$ as the solution of 
\begin{equation}\label{Eq:Main}
\begin{cases}
\partial_tu_{y}-\Delta u_{y}=f(t,x,u_{y})+y&\text{ in }\OT\,, 
\\ \mathcal Bu_{y}=0&\text{ on }(0;T)\times \partial \O\,, 
\\ u_{y}(0,\cdot)=u_0&\text{ in }\O,\end{cases}\end{equation}
where $T>0$ is a fixed time horizon. By the standard theory for non-linear parabolic equations \cite{Lieberman}, for any initial condition $u_0\in L^\infty(\O)$ and any source $y\in L^1(0,T;L^1(\O))\cap L^\infty(\OT)$, there exists a unique solution $u_y$ to \eqref{Eq:Main}.

Our goal is to optimise a fairly general class of criteria with respect to the source term $y$. It should be noted that \eqref{Hyp:f} is a loose enough set of technical assumptions to cover classical reaction terms of the form $f(t,x,u)=u(m(t,x)-u)$ where $m$ is a smooth function, which corresponds to monostable models, or $f(t,x,u)=u(u-\theta(t,x))(1-u)$, which models the Allee effect.

\subsection{Setting of the optimal control problem}
\paragraph{Cost functional}
 We fix two cost functions $j_1=j_1(t,x,u)\,, j_2=j_2(x,u)$ and define 
 \[ J:L^\infty(\OT)\ni y\mapsto \iint_\OT j_1\left(t,x,\uu(t,x)\right)dxdt+\int_\O j_2\left(x,\uu(T,x)\right)dx.\] This is the functional that is to be optimised, and we thus need to define the class of admissible controls we work with. As is often the case in applications \cite{Mazari_2021}, we enforce two constraints, an $L^\infty$ one and an $L^1$ one. In other words we consider three constants $0\leq\kappa_0\,, \kappa_1$ and $V_0\in (0;1)$, and we define the class of admissible controls  as 
 \begin{multline}\label{Eq:AdmY}\tag{$\bold{Adm}$}\mathcal Y:=\left\{ y\in L^1(0,T;L^1(\O))\,, -\kappa_0\leq y\leq \kappa_1\text{ almost everywhere in $\OT$ }\right.\\ \left.\text{ and, for almost every $t\in (0;T)$, }\fint_\O y(t,\cdot)=V_0\right\}.\end{multline}The symbol $\fint$ denotes the mean value of a function: $\fint_\O f=\frac1{\mathrm{Vol}(\O)}\int_\O f$.

The optimisation problem under consideration is: \begin{equation}\tag{$\bold{P}$}\label{Eq:PvIntro2}\fbox{$\displaystyle\max_{y\in \mathcal Y}\,J(y)$}\end{equation}

\paragraph{Regularity assumptions}
We work under the following assumptions on the cost functions $j_1\,, j_2$:
\begin{equation}\tag{$\bold{H}_{\mathrm{reg}}$}\label{Hyp:j}
j_1,j_2 \text{ are $\mathscr C^2$ in $[0;T]\times \overline \O\times \R$}.
\end{equation} 
Natural examples in the field of population dynamics would be $f=f(u)=u(u-\theta)(1-u)$ (bistable nonlinearity), $j_1=0$ and $j_2= u$, which corresponds to optimising a proportion of sane mosquitoes within a global population \cite{Almeida_2022,Mazari_2021}. Also note that the regularity assumptions on $j_1\,, j_2$ are far from minimal, and could be relaxed to being measurable in $x$ only. For the sake of readability, we describe our results under the stronger assumption \eqref{Hyp:j}.

 \paragraph{Optimality conditions for \eqref{Eq:PvIntro2}} 
 Let us describe the optimality conditions and adjoint state for \eqref{Eq:PvIntro2};  we have the following lemma, easily obtained from adapting \cite[Lemma 3]{Nadin_2020}:
\begin{lemma}\label{Le:DifferentiabilityK}
Under the assumptions \eqref{Hyp:f}-\eqref{Hyp:j}, the control-to-state map $T:\mathcal Y\ni y\mapsto u_y$ is twice Gateaux-differentiable at $y$. For any $y\in \mathcal Y$, for any  perturbation $h\in L^\infty(\OT)$, the first order Gateaux-derivative of the functional $J$ at $y$ in the direction $h$  is given by
\begin{equation}\label{Eq:DK1}\dot J(y)[h]=\iint_\OT p_y(t,x)h(t,x)dxdt\end{equation}
where $p_y$ is the solution of the backwards equation
\begin{equation}\label{Eq:AdjointK}
\begin{cases}
\partial_t p_y+\Delta p_y=-\partial_u j_1(t,x,u_y)-\partial_uf(t,x,u_y)p_y&\text{ in }(0;T)\times \O\,, \\
\mathcal Bp_y=0&\text{ on }(0;T)\times \partial \O\,, 
\\ p_y(T,\cdot)=\frac{\partial j_2}{\partial u}(x,u_y)&\text{ on }\partial\O.
\end{cases}
\end{equation}
Similarly, the second order Gateaux-derivative of the functional $J$ at $y$ in the direction $(h,h)$  is given by 
\begin{equation}\label{Eq:DDK2}
\ddot J(y)[h,h]=\iint_{(0;T)\times \O}\duy^2\left( p_y\frac{\partial^2 f}{\partial u^2}(t,x,u_y)
+\frac{\partial^2 j_1}{\partial u^2}(t,x,u_y)\right)dtdx+\int_\O \duy^2(T,x)\frac{\partial^2 j_2}{\partial u^2}(x,u_y)dx, 
\end{equation}
where $\duy$ is the unique solution of the linearised system
\begin{equation}\label{Eq:DotuK}\begin{cases}
\partial_t\duy-\Delta\duy=h+\partial_uf(t,x,u_y)\dot u_y&\text{ in }(0;T)\times \O\,, \\
 \mathcal B\duy=0&\text{ on }(0;T)\times \partial \O\,, \\ 
 \duy(0,\cdot)=0&\text{ in }\O.\end{cases}
\end{equation}
\end{lemma}
The solution $p_y$ of \eqref{Eq:AdjointK} is called the \emph{adjoint} of \eqref{Eq:PvIntro2}.  It encodes the first order optimality conditions for \eqref{Eq:PvIntro2}, as shown by the following result, adapted from \cite[Theorem 2.1]{Nadin_2020}:

\begin{proposition}\label{prop:1st}Let $y^*$ be a solution of \eqref{Eq:PvIntro2}. 
Then there exists a measurable function $c:[0;T]\to \R$ such that, for almost every $t\in [0;T]$, 
$$\begin{cases} y^*(t,x) = \kappa_{1} \quad &\hbox{ if } \quad p_{y^*}(t,x)>c(t),\\
y^{*}(t,x) = -\kappa_{0} \quad & \hbox{ if } \quad p_{y^*}(t,x)<c(t),\\
\{p_{y^*}(t,\cdot)=c(t)\}\subset \{-\kappa_0<y^*(t,\cdot)<\kappa_1\}.
\end{cases}$$
where $p_{y^*}$ is the unique solution of \eqref{Eq:AdjointK} with $y=y^*$.
\end{proposition} 

While already containing several extremely valuable information about the values of the optimal control, the first-order optimality conditions given in Proposition \ref{prop:1st} can prove complicated to handle, both from a theoretical and analytical point of view. Indeed, if we consider the second-order optimality conditions, it appear that if $f\,, j_1\,, j_2$ are convex, then the map $J$ itself is convex, so that all optimiser are of ``bang-bang" type: they are extreme points $y^*$ of the admissible class $\mathcal Y$, which, as easily checked, write $-\kappa_0 \mathds 1_E+\kappa_1\mathds 1_E$ for some measurable subset $E$ of $\O$. How to go beyond such convexity assumptions? Indeed, in many applications \cite{Almeida_2022} $f$ is neither convex, nor concave. In this situation, difficulties arise: while bang-bang controls  are generically expected to occur if the control problem is bilinear \cite{MNP2021} rather than linear, there is no way to prohibit \emph{a priori} the existence of an abnormal zone $\{-\kappa_0<y^*<\kappa_1\}$  where the optimisers do not saturate the constraint. Is it possible to give finer properties of $u_{y^*}$ on this abnormal set?
%
 Essentially, it can be proved  that this optimality conditions entail the existence of a function $\tilde f$ such that, for any optimiser $y^*$, we have, on $\{-\kappa_0<y^*<\kappa_1\}$, an equation of the form $-\partial_t p_y=\partial_u f(t,x,u_y)c(t)+\partial_uj_1(u)$. This is easily adapted from the analysis  of \cite[Section 5-Numerical Algorithm]{Nadin_2020}. Thus if we define $W=W(t,x,u)=\partial_u f(t,x,u_y)p_y+\partial_uj_1(t,x,u_y)$ the equation to be solved is of the form $Z=\ell(t,x)$ for some function $\ell$.
However when $Z$ is neither concave nor convex in $u$, this equation can have multiple roots. 
As was observed in \cite{Mazari_2021,Nadin_2020}  when optimising with respect to the initial condition, this is problematic when dealing with numerical approximations of the problem: which root should we choose? A good way to lift the ambiguity in the choice of the root is to use second-order optimality conditions. Of course, the main difficulty with the way $\ddot J$ is written in Lemma \ref{Le:DifferentiabilityK} is that the expression is distributed over $\OT$ while we would need a localised information, at any (or almost every) $t$.  If we assume for the sake of simplicity that 
$j_2=0$ and if we set $W:=\partial_u f(t,x,u_y)p_y+\partial_uj_1(t,x,u)$, we have 
\[ \ddot J(y^*)[h,h]=\iint_\OT \dot u_{y^*}^2(\partial_uW)\leq 0\] for any $h$. Can we deduce that at an optimiser $y^*$, when solving the equation $W=\ell$, we must choose a root that satisfies $\partial_uW\leq 0$?  This can prove to be very challenging, but the answer is positive, and it is the purpose of Theorem \ref{Th:Main2}; we underline that this is a natural question in the context of optimal control of semilinear equations from the theoretical point of view, as well as from the numerical one. Before we present our result, let us give some more details about the numerical approximation of \eqref{Eq:PvIntro2}.

Another related query is: can we localise (in time) the bang-bang property? In other words, is it true that if, at some $t_0>0$, both $f(t,x,\cdot)$ and $j_1(t,x,\cdot)$ are convex, then any optimal control $y^*$ satisfies that $y^*(t_0,\cdot)$ is a bang-bang function? Our result 	also provides an answer to this question.

\subsubsection{Main results about second order optimality conditions}
We assume that the non-linearity $f$ and the cost functions $j_1\,, j_2$ satisfy \eqref{Hyp:j}.
Our main result is the following:
\begin{theorem}\label{Th:Main2} Assume $f\,, j_1\,, j_2$ satisfy \eqref{Hyp:f}-\eqref{Hyp:j}. Let $y^*$ be a maximiser of $J$ over $\mathcal{Y}$. Let $\omega:=\{-\kappa_0<y^*<\kappa_1\}$ be the so-called abnormal set, and assume that $\mathrm{Vol}(\omega)>0$. Then 
\[ Z_{y^*}\leq 0\text{ a.e. in }\omega\]
where for any $y\in \mathcal Y$:
$$Z_{y}(t,x) :=p_y\frac{\partial^2f}{\partial u^2}(t,x,u_y)+\frac{\partial^2j_1}{\partial u^2}(t,x,u_y).$$
\end{theorem}
\begin{remark}
Enforcing the constraint $\fint_\O y(t,\cdot)=V_0$ rather than a constraint of the type $\fint_\O y(t,\cdot)=V_0(t)$, or $\iint_\OT y=V_0$ is immaterial to our analysis, and the conclusions of Theorem \ref{Th:Main2} can immediately be adapted to these case.
\end{remark}
The following corollary immediately follows from Theorem \ref{Th:Main2} but is somehow unexpected, and exemplifies the intricate behaviour of optimal control problems:
\begin{corollary}
Assume $j_2(x,\cdot)$ is increasing in $u$ for any $x$. Let $y^*$ be a solution of \eqref{Eq:PvIntro2}. For any $t\in (0;T)$ such that $f(t,x,\cdot)$ and $j_1(t,x\cdot)$ are convex in $u$, with either one of them strictly convex in $u$, $y^*(t,\cdot)$ is bang-bang.
\end{corollary}
In other words, \emph{the bang-bang property is fully localised in time}.

\subsection{A Laplace-type method}

To prove Theorem \ref{Th:Main2}, we rely on a new technique, which we dub a Laplace-type method. This is a combination of the technique we developed with Toledo in \cite{Mazari_2021}, which relied on Laplace-type arguments for a simple perturbation of the initial datum which is only well-fitted on interior points of the so-called \textbf{abnormal set }$\{-\kappa_0<y^*<\kappa_1\}$, and of the technique developed by the authors and Privat in \cite{MNP2021} in another framework, in order to construct suitable perturbations regardless of any regularity assumption on the abnormal set. Note that \cite{Mazari_2021} works in one-dimension only, and that it is intricate to extend the method to higher dimensions.

\paragraph{Statement of the result}
We consider the sequence of eigenvalues $\{\lambda_{k,\mathcal B}\}_{k\in \N}$, associated with eigenfunctions $\{\p_{k,\mathcal B}\}_{k\in \N}$ of the Laplace operator:
\begin{multline}\label{Eq:DefEvEf}
\begin{cases}-\Delta \p_{k,\mathcal B}=\lambda_{k,\mathcal B}\p_{k,\mathcal B}&\text{ in }\O\,, 
\\ \mathcal B \p_{k,\mathcal B}=0&\text{ on }\partial \O\,, 
\\ \int_\O \p_{k,\mathcal B}^2=1.\end{cases}\end{multline}
We order the eigenvalues in increasing order: 
$$0\leq \lambda_{1,\mathcal B}\leq \lambda_{2,\mathcal B}\leq \dots\leq \lambda_{k,\mathcal B}\underset{k\to \infty}\rightarrow \infty.$$

We are now in position to state our main technical result.

\begin{theorem}\label{Th:LaplaceDimd} 
    Let $q\in L^\infty(\OT)$ be a fixed potential and let $\omega\subset \O$ be a closed subset of $\O$ with positive measure. We assume that 
 for any $r\in [1;+\infty)$ there holds
    \begin{equation}\label{Hyp:q}\tag{$\bold{H}_q$}\partial_t q-\Delta q\,, \n q\in L^r(0,T;L^r(\O)).\end{equation} Additionally, if $\mathcal B$ is of Neumann type, we assume that $\mathcal B q=0$.
We consider a sequence $(h_K)_{K\in \N}\in L^2(\O)^\N$ such that, for any $K\in \N$, $h_K$ writes \[ h_K:=\sum_{k=K}^\infty a_{K,k}\p_{k,\mathcal B},\] where the sequence   $(a_{K,k})_{k\in \N}$ satisfies \[ \sum_{k=K}^\infty a_{K,k}^2=1,\]  and such that 
\[ \mathrm{supp}(h_K)\subset \omega \text{ in the sense that }h_K\mathds 1_\omega=h_K.\]
Define $v_K$ as the solution of the heat equation 
\[\begin{cases}\frac{\partial v_K}{\partial t}-\Delta v_K=qv_K&\text{ in }(0;T)\times \O\,, 
\\ \mathcal B v_K=0&\text{ on }(0;T)\times \partial \O,
\\ v_K(0,\cdot)=h_K&\text{ in }\O.
\end{cases}
\] Consider the unit ball $X$ of the space of Radon measures on $[0;T]\times \overline \O$.
Finally, define a sequence of probability measures $\left\{\nu_K\right\}_{K\in \N}\in X^\N$ by 
\[ \forall K \in \N\,, \nu_K:=\frac{v_K^2}{\iint_\OT v_K^2}.\] Then any closure point $\nu_\infty\in X$ of the sequence $\left\{\nu_K\right\}_{K\in \N}\in X^\N$ satisfies
\[\mathrm{supp}(\nu_\infty)\subset \{t=0\}\times\omega \,, \iint_\OT \nu_\infty=1\,, \nu_\infty\geq 0\text{ in the sense of measures}. \]
\end{theorem}

In this last equality, $\mathrm{supp}(\nu_\infty)\subset \{t=0\}\times\omega $ should be understood as follows: for any $\p\in \mathscr C^0([0;T]\times \overline\O)$ such that $\mathrm{supp}(\p)\subset \left(\{t=0\}\times \omega\right)^c$, $\langle \nu_\infty,\p\rangle=0$ where $\langle \cdot,\cdot\rangle$ stands for the duality bracket on $\mathscr C^0([0;T]\times\overline\O)$.


\paragraph{Regarding our terminology}
In this paragraph we justify the terminology of the title of this paper. First of all, we claim that Theorem \ref{Th:LaplaceDimd} is an extension of the standard two-scales expansion technique for parabolic equations. Namely, consider, as done in \cite{Mazari_2021}, the solution $w_K$ of the equation
\begin{equation}\label{Eq:wk}
\begin{cases}
\partial_t w_K-\partial_{xx} w_K=qw_K&\text{ in }(0;T)\times \T\,, 
\\ w_K(0,x)=\theta (x)\sin(Kx)&\text{ in }\T,
\end{cases}
\end{equation} where $\T$ is the one-dimensional torus and $\theta$ is a smooth bump function in $\T$. In {\cite{Mazari_2021}} it is proved that 
\[ w_K\underset{K\to \infty}\sim \theta(x)\sin(Kx)e^{-K^2t}\] in the $L^2((0;T)\times \T)$ sense. Consequently, using the fact that 
\[ \sin(Kx)^2\underset{K\to\infty}\rightharpoonup \frac12\] and the Laplace method, this implies that, for any smooth test function $\phi$, 
\[ \iint_{(0;T)\times \O}w_K^2\phi\underset{K\to \infty}\sim \frac1{2K^2} \int_\T  \theta(x)^2\phi(0,x)^2dx.\] In other words, in the limit $K\to \infty$, we only see (up to a proper rescaling) the support of the initial condition. In higher dimensional situations, Theorem \ref{Th:LaplaceDimd} establishes the same kind of qualitative behaviours, but we highlight here several non-trivial difficulties. First and foremost, it is not true in general that $\p_{K,\mathcal B}^2\underset{K\to \infty}\rightharpoonup \frac12$, since in many domains we may have a so-called localisation phenomenon \cite{Grebenkov_2013}. Second, this type of expansion only holds under strong regularity assumptions on the function $\theta$. In particular, this result assumes, in a sense, that we are considering highly oscillating initial conditions, with a regular support (for instance, that has non-empty interior). When considering applications to the optimal control of reaction-diffusion equations, it is extremely difficult (and, in general, a completely open question) to obtain this type of regularity.

Regarding the first difficulty, as a byproduct of the proof of Lemma \ref{Le:L2Estimate} below, we obtain that
\[ v_K\underset{K\to \infty}\sim \sum_{k=K}^\infty a_{K,k}\p_{k,\mathcal B}e^{-t\lambda_{k,\mathcal B}}.\] This is not yet enough to conclude as to the support in space of the limit $\nu_\infty$ as this would only yield, for any smooth function $\phi$, 
\[ \langle v_K^2,\phi\rangle\underset{K\to \infty}\sim \sum_{k,k'=K}^\infty \frac{a_{K,k}a_{K,k'}}{\lambda_{k,\mathcal B}+\lambda_{k',\mathcal B}}\int_\O \phi(0,\cdot)\p_{k,\mathcal B}\p_{k',\mathcal B},
\] and it is then unclear from this expression to derive a meaningful information about the support of $\nu_\infty$.

\subsection{Bibliographical references}\label{Se:Prior}
 
 We  investigated a related optimisation problems in an earlier paper with Toledo \cite{Mazari_2021}. More precisely, we were, rather than optimising with respect to the internal control $y$, trying to optimise a criterion with respect to the initial condition $u_0$, under the constraints $0\leq u_0\leq 1\,, \int_\O u_0=V_0$.  The results of \cite{Mazari_2021} were set in the case $d=1$ with periodic boundary conditions; several types of non-linearities were considered. First, the case where $f$ only depended on $u$ and was convex, with $\kappa_{0}=0$, $\kappa_{1}=1$, $f(0)=f(1)=0$, $j_{1}\equiv 0$ and $j_{2}(x,u)\equiv u$.  We have proved that in that case $u_{0}^{*}\equiv 1_{(0,V_{0})}$ is a global maximiser of $J$. Apart from this example, it is not true in general that the optimal initial conditions are bang-bang. Indeed, if $f$ is concave in $u$, then the second author and Toledo proved \cite{Nadin_2020} that the constant function $u_{0}^{*}\equiv V_{0}\in (0,1)$ is a global maximiser. Thus optimal controls are not always bang-bang. Similar results where derived when  $j_{2}(x,u)\equiv -(1-u)^{2}$ in \cite{Elie} and it is fairly straightforward to see that for internal linear control problems, optimisers can also fail to be bang-bang. This emphasised the need for understanding the behaviour of the maximiser $u_{0}^{*}$ on the abnormal set $\{-\kappa_0<u_0^*<\kappa_1\}$. Thus, the second case covered in \cite{Mazari_2021} made no convexity assumptions on $f$. We proved with Toledo that in the one-dimensional case, for any interior point $x$ of $\{-\kappa_0<u_0^*<\kappa_1\}$, one has $f''\big(u_{0}^{*}(x)\big)\leq 0$.  From there on, it is natural to both consider the case of internal controls, as is the case here, to go beyond the case of open abnormal sets and to higher dimensional situations. Two ways are available: the first one is to establish \emph{a priori} regularity for the abnormal set. However, the question of regularity of optimal controls is a difficult one, and we can not rule out that the interior of the abnormal set $\{-\kappa_0<u_0^*<\kappa_1\}$ is empty. Regularity issues in the study of optimal control problem is a major challenge, and, so far, most available results deal with the case of energetic functional in bilinear optimisation \cite{CKT}.
 The other one, which we take here, is to introduce a new type of methods to handle the case of merely measurable abnormal set.
 
  So we want to derive a result that holds almost everywhere on $\{-\kappa_0<y^*<\kappa_1\}$, and not only on its interior. One of the reasons such information are important is the numerical approximation of these $L^\infty-L^1$-constrained optimal control problems, a standard and powerful algorithm is the \emph{thresholding scheme}, akin to a gradient ascent method. Roughly speaking, it is expected that optimisers $u_0^*$ can be described using the level-sets of the so-called "adjoint state". When optimisers $y^*$ are bang-bang, it is expected that this scheme can be defined and used with the knowledge of first order optimality conditions only. That an optimiser $y^*$ is not bang-bang essentially amounts to saying that the adjoint state $p_{y^*}$ has a level-set of positive measure, which leads to using second-order optimality conditions in the definition of this scheme. Thus, having tractable information about the behaviour of optimisers $y^*$ in the set $\{-\kappa_0<y^*<\kappa_1\}$ is essential in implementing a cost-efficient algorithm. Finally, let us mention the recent \cite{Abdul_Halim_2022}, in which the same problem is discussed from another qualitative point of view: the authors study the influence of adding advection terms to the main equation on the value of the functional to optimise.

In order to further characterise $y^{*}$ on the abnormal set $\{-\kappa_0<y^*<\kappa_1\}$, one needs to extract information from the first and second order optimality conditions. Let us now explain why we could not use earlier results on optimal control for parabolic equations and what our contribution to this field of research is.  
There is a vast literature on this topic, and we will only focus here on earlier works that are close to the problem we consider here, that is, second order optimality conditions for a control on the initial datum.  For a general introduction to the optimal control of partial differential equations we refer to the book \cite{Troelzsch}.

First order optimality conditions for semi-linear parabolic equations, essentially encoded in the Pontryagin maximum principle, have been established in a very general framework in \cite{RaymondZidani}. In this paper, three types of controls are considered: one acts on the initial datum, one acts as a source term in $(0,T)\times \O$, as in the present article, and one acts on the boundary $(0,T)\times \partial \O$. A major difference with the present paper is that $L^{1}$ constraints are not covered by their framework. Here, we consider a simpler problem, since our control only acts as a source term. The reason for this is that we want to isolate the phenomenon we exhibit.

Sufficient second order conditions guaranteeing local optimality have been discussed  in a variety of situation when the control acts on $(0,T)\times \O$  and/or on the boundary $(0,T)\times \O$ (see \cite{CasasT, RaymondT, RT}). Let us also mention a wide literature on second order conditions for optimal control of semi-linear elliptic equations (see for example \cite{CasasReyesT}). 
The general approach of these papers is to derive the necessary second order optimality condition $\ddot J [u_{0}]\leq 0$, and to provide sufficient conditions in order to characterise a local maximiser. A Hamiltonian $H=H(t,x,u,p)$ is often derived from the first order conditions (see \cite{RaymondZidani}), and the second order necessary optimality conditions are described in terms of the hessian of the hamiltonian.  It is unclear how to penalise the time localised constraint $\fint_\O y(t,\cdot)=V_0$.
It seems extremely challenging to extract any information from second order optimality conditions using these earlier approaches in that case. More generally, we believe that these earlier works are not well-fitted to $L^{1}$ constraints. In the present, we push further the second order optimality conditions using a Laplace-type method that allows to concentrate the relevant information at any fixed $t$. We do not investigate sufficient conditions and leave it for a future work.

\section{Proof of Theorem \ref{Th:LaplaceDimd}}
We begin with the proof of Theorem \ref{Th:LaplaceDimd}, as Theorem \ref{Th:Main2} is a corollary of it.
\subsection{Steps of the proof}

The proof is divided up in several steps. As each can be technical and sometimes long, we summarise them here:
\begin{itemize}
\item First  we give some basic preliminary results related to parabolic regularity and the Laplace method. We refer to Propositions \ref{Pr:ParabolicRegularity} and \ref{Pr:Laplace}.
\item Second, we prove an estimate of the $L^2$-norm of $v_K$ under the form 
\[ \iint_\OT v_K^2\geq c_0\sum_{k=K}^\infty \frac{a_{K,k}^2}{\lambda_{k,\mathcal B}}.\] We refer to Lemma \ref{Le:L2Estimate} below.
\item Third, we prove that $\mathrm{supp}(\nu_\infty)\subset \{t=0\}\times \O$, see Lemma \ref{Le:TimeSupport}.
\item Finally, we prove that $\mathrm{supp}(\nu_\infty)\subset [0;T]\times \omega$, thereby concluding the proof.
\end{itemize}

\subsection{Step 1: Preliminaries on the parabolic regularity and the Laplace method}

\paragraph{A preliminary parabolic regularity result}
We recall the following parabolic regularity result:
\begin{proposition}\label{Pr:ParabolicRegularity}
Let $q\,, g\in L^\infty(\OT)$. For any $\theta_0\in L^\infty(\O)$, the solution $\theta$ of 
\begin{equation}\label{Eq:Theta}\begin{cases}\partial_t\theta-\Delta \theta-q\theta=g&\text{ in }\OT\,, 
\\ \mathcal B\theta=0&\text{ on }(0;T)\times \partial \O\,, 
\\ \theta(0,\cdot)=\theta_0\end{cases}\end{equation} satisfies
\[ \sup_{t \in [0;T]} \Vert \theta(t,\cdot)\Vert_{L^2(\O)}\leq C\left(\int_0^T \Vert g(t,\cdot)\Vert_{L^2(\O)}dt+\Vert \theta_0\Vert_{L^2(\O)}\right)\] where the constant $C$ only depends on $\Vert q\Vert_{L^\infty(\OT)}$, $\mathrm{Vol}(\O)$ and $T$.\end{proposition}
As this result is instrumental in deriving our estimates we prove it here.
\begin{proof}[Proof of Proposition \ref{Pr:ParabolicRegularity}]
Multiplying \eqref{Eq:Theta} by $\theta$ and integrating by parts in time we obtain 
\[ \frac{d}{dt}\int_\O \frac{\theta(t,\cdot)^2}{2}+\int_\O |\n \theta|^2-\Vert q\Vert_{L^\infty(\OT)}\int_\O \theta^2\leq \int_\O g\theta\leq \Vert g(t,\cdot)\Vert_{L^2(\O)}\Vert \theta(t,\cdot)\Vert_{L^2(\O)}\] whence 
\[ \frac{d}{dt}(\Vert \theta(t,\cdot)\Vert_{L^2(\O)})-\Vert q\Vert_{L^\infty(\OT)}\Vert \theta(t,\cdot)\Vert_{L^2(\O)}\leq \Vert g(t,\cdot)\Vert_{L^2(\O)}.\] It suffices to apply the Gr\"onwall lemma to conclude.
\end{proof}

\paragraph{Background on the Laplace method} We recall here the following result:
\begin{proposition}\label{Pr:Laplace}
For any $m\in \N$, 
\[ \int_0^T t^m e^{-k t}dt\sim_{k\to \infty}\frac{C_m}{k^{m+1}}.\]
\end{proposition}
\begin{proof}[Proof of Proposition \ref{Pr:Laplace}]
Integrating by parts $m$-times we have 
\[ \int_0^T t^m e^{-kt}dt=\frac{m!}{k^{m+1}}(1-e^{-kT})\] whence the conclusion.
\end{proof}

\subsection{Step 2: Asymptotic of the $L^2$ norm of the solution}
The goal of this paragraph is to prove the following result:
\begin{lemma}\label{Le:L2Estimate}
There exists a constant $c_0>0$ such that
\[ \forall K\in \N\,, \iint_\OT v_K^2\geq c_0\sum_{k=K}^\infty\frac{a_{K,k}^2}{\lambda_{k,\mathcal B}}.\]
\end{lemma}
\begin{proof}[Proof of Lemma \ref{Le:L2Estimate}]
Let us introduce, for any $k\in \N$, the function
\[ w_{0,K}(t,x):=\sum_{k=K}^\infty a_{K,k}\p_{k,\mathcal B}e^{-\lambda_{k,\mathcal B}t}.\] This function solves
\[ \left(\partial_t-\Delta\right)w_{0,K}=0.\]
It is expected that there should hold 
\begin{equation}\label{Heights} v_K\approx w_{0,K}\end{equation} in a certain sense. In order to formalise \eqref{Heights} we first compute explicitly
\begin{equation}\label{Eq:EstW0K} \iint_\OT w_{0,K}^2=\sum_{k=K}^\infty \frac{a_{K,k}^2}{2\lambda_{k,\mathcal B}}(1-e^{-2T\lambda_{k,\mathcal B}})\geq \frac14\sum_{k=K}^\infty \frac{a_{K,k}^2}{\lambda_{k,\mathcal B}}\end{equation}whenever $K$ is large enough to ensure that $1-e^{-2T\lambda_{K,\mathcal B}}\geq \frac12$. To control the distance between $w_{0,K}$ and $v_K$, consider the remainder term
\[ T_{0,K}:= v_K- w_{0,K}.\] It is clear that $T_{0,K}$ satisfies
\[ \partial_t T_{0,K}-\Delta T_{0,K}-qT_{0,K}=qw_{0,K}.\] But this is not yet enough. Indeed, if we were to apply Proposition \ref{Pr:ParabolicRegularity} directly, we would need to estimate
$\int_0^T\Vert qw_{0,K}\Vert_{L^2(\O)}$ but we can \emph{a priori} only bound it as 
\[ \int_0^T\Vert qw_{0,K}\Vert_{L^2(\O)}\leq \Vert q\Vert_{L^\infty(\OT)}\left(\sum_{k=K} \frac{a_{K,k}^2}{\lambda_{k,\mathcal B}}\right)^{\frac12},\] that is, by a term of order $\Vert w_{0,K}\Vert_{L^2(\OT)}$, which is not strong enough. We thus have to take more care when handling this. For this reason, introduce the function 
\[ z_{0,K}:\OT\ni (t,x)\mapsto tq(t,x)w_{0,K}(t,x)\] and define 
\[ R_{0,K}:= v_K-w_{0,K}-z_{0,K}.\]
As $z_{0,K}$ satisfies
\[ \partial_t z_{0,K}-\Delta z_{0,K}=qw_{0,K}+(\partial_tq-\Delta q) (tw_{0,K})-2t\langle \n q,\n w_{0,k}\rangle
\]
we obtain 
\[ \partial_tR_{0,K}-\Delta R_{0,K}-qR_{0,K}=\underbrace{qz_{0,K}}_{=:V_{0,K}}-\underbrace{(\partial_tq-\Delta q) (tw_{0,K})}_{=:V_{1,K}}+2\underbrace{t\langle \n q,\n w_{0,k}\rangle}_{=:V_{2,K}}.\]
Moreover, notice that 
there holds 
\[ \mathcal Bz_{0,K}=0\] which in turn implies
\[ \mathcal B R_{0,K}=0.\] Finally, we have 
\[ R_{0,K}(0,\cdot)=0\text{ in }\O\] by construction. From Proposition \ref{Pr:ParabolicRegularity} there holds, for some constant $C>0$ independent of $K$, 
\begin{multline*} \sup_{t\in [0;T]}\Vert R_{0,K}(t,\cdot)\Vert_{L^2(\O)}\leq C\left(\int_0^T \Vert V_{0,K}(t,\cdot)\Vert_{L^2(\O)}dt+\int_0^T \Vert V_{1,K}(t,\cdot)\Vert_{L^2(\O)}dt\right.\\+\left.\int_0^T \Vert V_{2,K}(t,\cdot)\Vert_{L^2(\O)}dt\right).\end{multline*} 
Now observe that 
\begin{align*} \int_0^T \Vert V_{0,K}(t,\cdot)\Vert_{L^2(\O)}dt
&=\int_0^T t\Vert q(t,\cdot)w_{0,K}(t,\cdot)\Vert_{L^2(\O)}dt
\\&\leq \Vert q\Vert_{L^\infty(\OT)}\int_0^Tt \Vert w_{0,K}(t,\cdot)\Vert_{L^2(\O)}dt
\\&\leq \Vert q\Vert_{L^\infty(\OT)}\left(\iint_\OT t^2 w_{0,K}^2\right)^{\frac12}
\\&\leq \Vert q\Vert_{L^\infty}\sqrt{\sum_{k=K}^\infty a_{K,k}^2\int_0^T t^2e^{-2t\lambda_{k,\mathcal B}}dt}
\\&\leq \frac{M}{{\lambda_{K,\mathcal B}}}\sqrt{\sum_{k=K}^\infty \frac{a_{K,k}^2}{\lambda_{k,\mathcal B}}}.
\end{align*}
In the last step, we applied Proposition \ref{Pr:Laplace} with $m=3$. We have thus proved
 \begin{equation}\label{Eq:EstV0K}
\int_0^T \Vert V_{0,K}(t,\cdot)\Vert_{L^2(\O)}dt\leq \frac{C}{{\lambda_{K,\mathcal B}}}\sqrt{\sum_{k=K}\frac{a_{K,k}^2}{\lambda_{k,\mathcal B}}}.\end{equation}
Similarly, we can estimate $\int_0^T \Vert V_{1,K}(t,\cdot)\Vert_{L^2(\O)}dt$. Define $Q:=\partial_tq-\Delta q$. Then there holds
\begin{align*}
\int_0^T \Vert V_{1,K}(t,\cdot)\Vert_{L^2(\O)}dt&=\int_0^T \Vert Q(t,\cdot) tw_{0,K}(t,\cdot)\Vert_{L^2(\O)}dt
\\&\leq C\int_0^T t \Vert Q(t,\cdot)\Vert_{L^{r_0}(\O)}\Vert w_{0,K}(t,\cdot)\Vert_{L^{p_0}(\O)}
\\&\text{ from  the H\"{o}lder inequality with $1/{r_0}+1/{p_0}=1/2$}
\\&\leq C\sqrt{\int_0^T \Vert Q(t,\cdot)\Vert_{L^{r_0}(\O)}^2 }\sqrt{\int_0^T t^2 \Vert  w_{0,K}\Vert_{L^{p_0}(\O)}^2}
\\&\text{ from the Cauchy-Schwarz inequality}.
\end{align*}
We now choose $p_0>2$ such that 
\[ W^{1,2}(\O)\hookrightarrow L^{p_0}(\O)\] and fix the corresponding exponent $r_0$. Then, up to a constant $C>0$ we have 
\begin{align*}
\int_0^T \Vert V_{1,K}(t,\cdot)\Vert_{L^2(\O)}dt&\leq C\sqrt{\int_0^T \Vert Q(t,\cdot)\Vert_{L^{r_0}(\O)}^2 }\sqrt{\int_0^T t^2 \Vert  w_{0,K}\Vert_{L^{p_0}(\O)}^2}
\\&\leq C\sqrt{\int_0^T \Vert Q(t,\cdot)\Vert_{L^{r_0}(\O)}^2 }\sqrt{\int_0^T t^2 \Vert  \n w_{0,K}\Vert_{L^{2}(\O)}^2}.
\end{align*}
Now observe that by the Jensen inequality we have, up to a constant  still denoted $C$ for notational convenience
\begin{equation}\label{Eq:Ajout1} \int_0^T \Vert Q(t,\cdot)\Vert_{L^{r_0}(\O)}^2dt\leq C\left(\iint_\OT |Q|^{2r_0}\right)^{\frac1{r_0}}=C\Vert Q\Vert_{L^{2r_0}(0,T;L^{2r_0})}^{4}=C_{r_0}<\infty.\end{equation} In the last inequality we used Assumption \eqref{Hyp:q}. All in all, up to a multiplicative constant once again denoted by $C$, we have obtained
\begin{align*}
\int_0^T \Vert V_{1,K}(t,\cdot)\Vert_{L^2(\O)}dt&\leq C \sqrt{\int_0^T t^2 \Vert  \n w_{0,K}\Vert_{L^{2}(\O)}^2dt}
\\&=C \sqrt{\sum_{k=K}^\infty \lambda_{k,\mathcal B}a_{K,k}^2\int_0^T t^2e^{-2t\lambda_{k,\mathcal B}}dt}
\\&\leq \frac{C}{\sqrt{\lambda_{K,\mathcal B}}}\sqrt{\sum_{k=K}\frac{a_{K,k}^2}{\lambda_{k,\mathcal B}}}
\end{align*}
We have thus obtained
\begin{equation}\label{Eq:EstV1K}
\int_0^T \Vert V_{1,K}(t,\cdot)\Vert_{L^2(\O)}dt\leq  \frac{C}{\sqrt{\lambda_{K,\mathcal B}}}\sqrt{\sum_{k=K}\frac{a_{K,k}^2}{\lambda_{k,\mathcal B}}}.
\end{equation}
Let us finally estimate $V_{2,K}$. Define $Q_1:=\n q$. We need to estimate
\begin{equation}\label{Eq:Ajout2}\textcolor{black}{
\int_0^T t\Vert\langle Q_1,\n w_{0,k}\rangle\Vert_{L^2(\O)}dt}.\end{equation}  Applying the same H\"{o}lder and Cauchy-Schwarz inequalities as above, we have 
\begin{align*}
\int_0^T t\Vert\langle Q_1,\n w_{0,k}\rangle\Vert_{L^2(\O)}dt&\leq \int_0^T t \Vert Q_1(t,\cdot)\Vert_{L^{r_0}(\O)}\Vert \n w_{0,K}(t,\cdot)\Vert_{L^{p_0}(\O)}dt
\\&\leq C \sqrt{\int_0^T t^2 \Vert \n w_{0,K}(t,\cdot)\Vert_{L^{p_0}(\O)}^2}
\end{align*}
where $1/{r_0}+1/{p_0}=1/2$.
 It remains to estimate the quantity
  \[ \int_0^T t^2\Vert\n w_{0,K}(t,\cdot)\Vert_{L^{p_0}(\O)}^2dt\] for some $p'>2$. However, by  the fractional Sobolev embedding \cite[Theorem 7.57]{Adams} (see also \cite[Theorem 3.4, Lemma 4.11]{chandler-wilde_hewett_moiola_2015}) $H^{1+\gamma}\hookrightarrow W^{1,p_0}(\O)$, for some $\gamma\in ]0;1[$  and $p_0>2$, we have 
\[ \Vert \n w_{0,k}(t,\cdot)\Vert_{L^{p_0}(\O)}^2\leq \Vert w_{0,k}(t,\cdot)\Vert_{H^{1+\gamma}(\O)}^2=\sum_{k=K}^\infty a_{K,k}^2 \lambda_{k,\mathcal B}^{1+\gamma}e^{-t\lambda_{k,\mathcal B}}\] so that the last term can be estimated as 
\[
\int_0^T t^2\sum_{k=K}^\infty a_{K,k}^2 \lambda_{k,\mathcal B}^{1+\gamma}e^{-t\lambda_{k,\mathcal B}}dt\sim_{K\to \infty}\sum_{k=K}^\infty\frac{a_{K,k}^2}{\lambda_{k,\mathcal B}^{2-\gamma}}
\]whence  we obtain 
\begin{equation}\label{Eq:EstV2K}\textcolor{black}{
\int_0^T t\Vert \langle \n Q_1,\n w_{0,k}\rangle\Vert_{L^2(\O)}dt=\int_0^T \Vert V_{2,K}(t,\cdot)\Vert_{L^2(\O)}dt\leq \frac{C}{\lambda_{K,\mathcal B}^{\frac{1-\gamma}2}}\sqrt{\sum_{k=K}^\infty \frac{a_{K,k}^2}{\lambda_{k,\mathcal B}}}}
\end{equation}
\begin{remark}\label{Re:Add}
It should be noted that the only property of the potential $Q_1=\n q$ we used to prove \eqref{Eq:EstV2K} was that $q$ satisfies \eqref{Hyp:q}.
\end{remark}
Summing estimates \eqref{Eq:EstV0K}-\eqref{Eq:EstV1K}-\eqref{Eq:EstV2K} we get that for some constant $C$ and some $\beta>0$ there holds 
\begin{equation}\label{Eq:EstR0K}
\sup_{t\in [0;T]}\Vert R_{0,K}(t,\cdot)\Vert_{L^2(\O)}\leq \frac{C}{\lambda_{K,\mathcal B}^\beta}\sqrt{\sum_{k=K}^\infty \frac{a_{K,k}^2}{\lambda_{k,\mathcal B}}}.
\end{equation}
Furthermore,  remembering that 
$z_{0,K}=tq(t,\cdot)w_{0,K}$ we have 
\begin{align}\label{Eq:ChainZ0}
\iint_\OT z_{0,K}^2&\leq \Vert q\Vert_{L^\infty(\OT)}\iint_\OT t^2 w_{0,K}^2dtdx
\\&=\Vert q\Vert_{L^\infty(\OT)}\sum_{k=K}^\infty \int_0^T t^2 a_{K,k}^2 e^{-2t\lambda_{K,k}}dt
\\&=2\Vert q\Vert_{L^\infty(\OT)}\sum_{k=K}^\infty \frac{a_{K,k}^2}{\lambda_{K,k}^3}
\\&=\underset{K\to \infty}o\left(\Vert w_{0,K}\Vert_{L^2(\OT)}^2\right).
\end{align}

 We turn back to the function $v_K$. Developing the square root we obtain
\begin{align*}
\iint_{\OT}v_K^2&=\iint_\OT w_{0,K}^2+\iint_\OT (R_{0,K}+z_{0,K})^2+2\iint_\OT w_{0,K}(R_{0,K}+z_{0,K})
\end{align*}
From \eqref{Eq:EstR0K}-\eqref{Eq:ChainZ0} and the algebraic inequality $|a+b|^2\leq 2(a^2+b^2)$ we deduce
\[ \iint_\OT (R_{0,K}+z_{0,K})^2\leq 2 \left(\iint_\OT R_{0,K}^2+\iint_\OT z_{0,K}^2\right)=\underset{K\to \infty}o\left(\Vert w_{0,K}\Vert_{L^2(\OT)}^2\right).\] Similarly, by the H\"{o}lder inequality,
\[ \iint_\OT w_{0,K}(R_{0,K}+z_{0,K})=\underset{K\to \infty}o\left(\Vert w_{0,K}\Vert_{L^2(\OT)}^2\right).\] Thus
\[\iint_\OT v_K^2=\iint_\OT w_{0,K}^2+\underset{K\to \infty}o\left(\Vert w_{0,K}\Vert_{L^2(\OT)}^2\right).\] As 
\[ \iint_\OT w_{0,K}^2\sim_{K\to \infty}\sum_{k=K}^\infty\frac{a_{K,k}^2}{\lambda_{k,\mathcal B}}\] the proof is finished.
\end{proof}

\subsection{Step 3: Controlling the support in time} The goal of this paragraph is the following lemma:
\begin{lemma}\label{Le:TimeSupport}
For any closure point $\nu_\infty$ of the sequence $\left\{\nu_K\right\}_{K\in \N}$ (defined in the statement of Theorem \ref{Th:LaplaceDimd}) there holds 
 \begin{equation}\label{Eq:TimeSupport} \mathrm{supp}(\nu_\infty)\subset \{t=0\}\times \O.\end{equation}
\end{lemma}
As we shall see, this is an almost straightforward consequence of the computations carried out in the proof of Lemma \ref{Le:L2Estimate}.
\begin{proof}[Proof of Lemma \ref{Le:TimeSupport}]
 From Lemma \ref{Le:L2Estimate} we know that for some constant $c_0>0$ we have 
 \[ \iint_{(0;T)\times \O}v_K^2\geq c_0\sum_{k=K}^\infty \frac{a_{K,k}^2}{\lambda_{k,\mathcal B}}.\] To prove \eqref{Eq:TimeSupport} it suffices to prove that, for any $\e>0$, 
 \[ \iint_{(\e;T)\times \O}v_K^2=\underset{K\to \infty}o\left( \sum_{k=K}^\infty \frac{a_{K,k}^2}{\lambda_{k,\mathcal B}} \right).\]  Using the same notations as in the proof of Lemma \ref{Le:L2Estimate} we have
 \begin{align*}
 \iint_{(\e;T)\times \O}v_K^2&=\iint_{(\e;T)\times \O}(v_K-w_{0,K}-z_{0,K})^2&(=:I_{1,K})
 \\&+2\iint_{(\e;T)\times \O} (v_{0,K}-w_{0,K}-z_{0,K}) (w_{0,K}+z_{0,K}) &(=:I_{2,K})
 \\&+\iint_{(\e;T)\times \O}({w_{0,K}}+z_{0,K})^2 &(=:I_{3,K}).
 \end{align*}
 As in the proof of Lemma \ref{Le:L2Estimate} we have 
 \[ I_{1,K}\,, I_{2,K}=\underset{K\to \infty}o\left( \sum_{k=K}^\infty \frac{a_{K,k}^2}{\lambda_{k,\mathcal B}} \right).\] It remains to estimate $I_{3,K}$. However, up to a multiplicative constant $C$ we have
 \begin{align*}
 I_{3,K}&=\iint_{(\e,T)\times \O}( w_{0,K}+z_{0,K})^2
 \\&\leq C \left(\iint_{(\e;T)\times \O}w_{0,K}^2+\iint_{(\e;T)\times \O}z_{0,K}^2\right)
 \\&\leq C \left(\iint_{(\e;T)\times \O}w_{0,K}^2+\iint_{(0;T)\times \O}z_{0,K}^2\right)
 \\&\leq C \left( \iint_{(\e;T)\times \O}w_{0,K}^2+\underset{K\to \infty}o(\Vert w_{0,K}\Vert_{L^2(\OT)}^2)\right)&\text{ from \eqref{Eq:ChainZ0}}
 \end{align*}
 Moreover, for a constant $C$
 \begin{align*}
 \iint_{(\e;T)\times \O} w_{0,K}^2&=\sum_{k=K}^\infty a_{K,k}^2\int_\e^T e^{-t\lambda_k}dt
 \\&\leq C e^{-\e \lambda_{K,\mathcal B}}\sum_{k=K}\frac{a_{K,k}^2}{\lambda_{k,\mathcal B}}
 \\&=\underset{K\to \infty}o\left(\Vert w_{0,K}\Vert_{L^2(\OT)}^2)\right)
 \end{align*}
 so that 
 \[ I_{3,K}=\underset{K\to \infty}o\left(\Vert w_{0,K}\Vert_{L^2(\OT)}^2\right).\] Summarising, we have obtained 
 \[ \iint_{(\e;T)\times \O} v_K^2=\underset{K\to \infty}o(\Vert w_{0,K}\Vert_{L^2(\OT)}^2)=\underset{K\to \infty}o\left(\Vert v_K\Vert_{L^2(\OT)}^2\right).\] Thus, for any test function $\phi\in \mathscr C^0_c([\e;T]\times \O)$, (the limit is taken along a subsequence)
 \begin{align*} \langle \nu_\infty,\phi\rangle&=\lim_{K\to \infty}\iint_\OT \nu_K\phi
 \\&=\lim_{K\to \infty}\iint_{(\e;T)\times \O} \nu_K \phi
 \\&\leq \Vert \phi\Vert_{L^\infty(\OT)}\lim_{K\to \infty}\frac{\iint_{(\e;T)\times \O} v_K^2}{\iint_\OT v_K^2}
 \\&=0.\end{align*}
 The conclusion follows.
\end{proof}

\subsection{Step 4: Controlling the support in space} The goal of this paragraph is the following result:
\begin{lemma}\label{Le:SpaceSupport}For any closure point $\nu_\infty$ of the sequence $\left\{\nu_K\right\}_{K\in \N}$ (defined in the statement of Theorem \ref{Th:LaplaceDimd}) there holds 
 \begin{equation}\label{Eq:SpaceSupport} \mathrm{supp}(\nu_\infty)\subset[0;T]\times \omega.\end{equation}\end{lemma}

\begin{proof}[Proof of Lemma \ref{Le:SpaceSupport}]
To prove \eqref{Eq:SpaceSupport} it suffices to prove the following: for any open set $F\subset \O$ such that $\mathrm{dist}(\overline F,\omega)>0$ (remember that $\omega$ is closed), for any $\phi \in \mathscr C^0([0;T]\times \O)$ such that for any $t$ $\phi(t,\cdot)\in \mathscr C^0_c(F)$, there holds 
\[ \langle \nu_\infty,\phi\rangle=0.\] Here $\nu_\infty$ is a closure point of the sequence $\left\{\nu_K\right\}_{K\in \N}$. Hence, fix an open set $F\subset\O$ such that $\mathrm{dist}(\overline F,\omega)>0$. We consider a smooth function $\theta\in \mathscr C^\infty_c(\O)$ such that 
\[ \theta h_K=h_K.\] This amounts to requiring that $\mathrm{supp}(h_K)\subset \{\theta=1\}$. Furthermore, we require that 
\[ \theta\equiv 0\text{ in }\overline F.\]

We now look for a two-scale like asymptotic expansion of the solution $v_K$ in terms of $\theta$. Introduce (with the notations of Lemma \ref{Le:L2Estimate})  
\[\eta_{0,K,\theta}:=\theta(x)\sum_{k=K}^\infty a_{K,k}\p_{k,\mathcal B}e^{-t\lambda_{k,\mathcal B}}=\theta(x)w_{0,K}(t,x)\] and
\[ R_{0,K,\theta}:=v_K-\eta_{0,K,\theta}.\] 
The function $R_{0,K,\theta}$ satisfies 
\[ \partial_t R_{0,K,\theta}-\Delta R_{0,K,\theta}-qR_{0,K,\theta}=2\langle \n \theta,\n w_{0,K}\rangle+(\Delta \theta)w_{0,K}+q\eta_{0,K,\theta}.\]
Define 
\[ G:=\Delta \theta+q\theta.\] The equation on $R_{0,K,\theta}$ rewrites 
\[ \partial_t R_{0,K,\theta}-\Delta R_{0,K,\theta}= 2\langle \n \theta, \n w_{0,K}\rangle+ G w_{0,K}.\]
We can hence split $R_{0,K,\theta}$ as 
\[ R_{0,K,\theta}=r_{1,K,\theta}+2r_{2,K,\theta}\] where 
\[
\begin{cases}
\partial_t r_{1,K,\theta}-\Delta r_{1,K,\theta}=G w_{0,K}&\text{ in }\OT\,, 
\\ \partial_t r_{2,K,\theta}-\Delta r_{2,K,\theta}=\langle \n \theta,\n w_{0,K}\rangle&\text{ in }\OT\,, 
\\ \mathcal B r_{j,K,\theta}=0&\text{ on }(0;T)\times \partial \O\,, \quad (j=1,2),
\\r_{j,K,\theta}(0,\cdot)=0&\text{ in }\O\,, \quad (j=1,2).
\end{cases}\]
We estimate $r_{1,K,\theta}$ and $r_{2,K,\theta}$ separately.
\paragraph{Estimate on $r_{1,K,\theta}$}
Introducing 
\[z_{1,K,\theta}:= tGw_{0,K}\] we show, exactly as in the proof of Lemma \ref{Le:L2Estimate}, that 
\[ \sup_{t\in [0;T]}\Vert r_{1,K,\theta}(t,\cdot)-z_{1,K,\theta}(t,\cdot)\Vert_{L^2(\O)}=\underset{K\to \infty}o\left(\sum_{k=K}^\infty \frac{a_{K,k}^2}{\lambda_{k,\mathcal B}}\right).
\]
Indeed, it suffices to observe that with the assumptions on $q$, and as $\theta\in\mathscr C^\infty_c(\O)$, $G$ also satisfies Assumption \eqref{Hyp:q}. Furthermore, for any $t\in [0;T]$, 
\begin{align*}
\Vert z_{1,K,\theta}(t,\cdot)\Vert_{L^2(\OT)}&\leq  \Vert G\Vert_{L^\infty(\O)}\sqrt{\sum_{k=K}^\infty a_{K,k}^2 \int_0^T t^2e^{-2t\lambda_{k,\mathcal B}}}
\\&=\underset{K\to \infty}o\left(\sqrt{\sum_{k=K}^\infty \frac{a_{K,k}^2}{\lambda_{k,\mathcal B}}}\right).
\end{align*}
Thus, 
\[ \Vert r_{1,K,\theta}\Vert_{L^2(\OT)}=\underset{K\to \infty}o\left(\sqrt{\sum_{k=K}^\infty \frac{a_{K,k}^2}{\lambda_{k,\mathcal B}}}\right).\]

\paragraph{Estimate on $r_{2,K,\theta}$}
Let us first reason heuristically.
Formally, we should have 
\[ r_{2,K,\theta}\approx t\langle \nabla\theta,\n w_{0,K}\rangle= \left\langle \nabla \theta,t\sum_{k=K}^\infty a_{K,k} \n \p_ke^{-t\lambda_{k,\mathcal B}}\right\rangle=:\tilde r_{2,K,\theta}.\] Let us first estimate $\tilde r_{2,K,\theta}$.
We have, up to a multiplicative constant $C$, 
\begin{align*} \int_0^T \Vert \tilde r_{2,K,\theta}\Vert_{L^2(\O)}dt&\leq C\left(\sum_{k=K}a_{K,k}^2\lambda_{k,\mathcal B}\int_0^T t^2e^{-2t\lambda_{k,\mathcal B}} dt\right)^{1/2}
\\&\leq C\left(\sum_{k=K}^\infty \frac{a_{K,k}^2}{\lambda_{K,k}^2}\right)^{\frac12}
\\&=\underset{K\to \infty}o\left(\sqrt{\sum_{k=K}^\infty \frac{a_{K,k}^2}{\lambda_{k,\mathcal B}}}\right).
\end{align*}
Consider now 
\[ r_K:=r_{2,K,\theta}-\tilde r_{2,K,\theta} \] The function $r_K$ satisfies
\[\partial_t r_K-\Delta r_K-qr_K=\underbrace{q\tilde r_{2,K,\theta}+\langle \n \Delta \theta,t\n w_{0,K}\rangle}_{=:J_K}+\underbrace{2t\left(\n^2\theta \odot \n^2 w_{0,K}\right)}_{=:I_K}
\] where $\odot$ denotes the Hadamard product of matrices. Adapting the computation that led to estimating $\tilde r_{2,K,\theta}$ we see that the solution $\beta_{1,K}$ of 
\[ \partial_t \beta_{1,K}-\Delta \beta_{1,K}-q\beta_{1,K}= J_K\] satisfies 
\begin{equation}\label{Eq:EstBeta1K}
\Vert \beta_{1,K}\Vert_{L^2(\OT)}=\underset{K\to \infty}o\left(\sqrt{\sum_{k=K}^\infty \frac{a_{K,k}^2}{\lambda_{k,\mathcal B}}}\right).
 \end{equation}
 Thus the only term that should be estimated is the solution $\tilde r_K$ of 
\[ \partial_t \tilde r_K-\Delta \tilde r_K-q\tilde r_K=t\left(\n^2\theta \odot \n^2 w_{0,K}\right) \] We introduce two last auxiliary functions, namely, 
\[ \tilde r_{3,K,\theta}:=\frac{t^2}2\left(\n^2\theta \odot \n^2 w_{0,K}\right)\,, \tilde T_K:=\tilde r_K-\tilde r_{3,K,\theta}.\]
On the one hand we have 
\[ \partial_t  \tilde T_K-\Delta \tilde T_K-q\tilde T_K=q\tilde r_{3,K,\theta}+\frac{t^2}2  \n^2\Delta \theta\odot \n^2 w_{0,K}+{t^2} \n \n^2 \theta\odot \n^2 \n w_{0,K}.
\]
On the other hand, up to a multiplicative constant $C$, \begin{align*}
\int_0^T t^2\Vert \n^2\theta \odot \n^2 w_{0,K}\Vert_{L^2(\O)}dt&\leq C \int_0^T t^2\Vert \theta\Vert_{\mathscr C^2(\O)}\Vert \n^2 w_{0,K}(t,\cdot)\Vert_{L^2(\O)}dt
\\&\leq C\int_0^T t^2 \Vert \Delta w_{0,K}(t,\cdot)\Vert_{L^2(\O)} dt\text{ by elliptic regularity}
\\&\leq C\left(\int_0^T t^4 \sum_{k=K}^\infty a_{K,k}^2\lambda_{k,\mathcal B}^2e^{-2t\lambda_{k,\mathcal B}}\right)^{\frac12}
\\&\leq C \left(\sum_{k=K}^\infty \frac{a_{K,k}^2}{\lambda_{k,\mathcal B}^3}\right)^{\frac12}
\\&=\underset{K\to \infty}o\left(\sqrt{\sum_{k=K}^\infty \frac{a_{K,k}^2}{\lambda_{k,\mathcal B}}}\right).
\end{align*}
Finally, up to a multiplicative constant, we have
\begin{align*}
\int_0^T t^2 \Vert \n \n^2 \theta\odot \n^2 \n w_{0,K}(t,\cdot)\Vert_{L^2(\O)}dt
&\leq  C\int_0^T t^2 \Vert \Delta \nabla w_{0,K}(t,\cdot)\Vert_{L^2(\O)}^2dt\text{ by elliptic regularity }
\\&\leq C\left(\int_0^T \sum a_{K,k}^2 t^4 \lambda_{k,\mathcal B}^3e^{-2t\lambda_{k,\mathcal B}}dt\right)^{1/2}
\\&\leq C \left(\sum_{k=K}^\infty \frac{a_{K,k}^2}{\lambda_{k,\mathcal B}^2}\right)^{\frac12}
\\&=\underset{K\to \infty}o\left(\sqrt{\sum_{k=K}^\infty \frac{a_{K,k}^2}{\lambda_{k,\mathcal B}}}\right).
\end{align*}
We can hence conclude that 
\begin{equation}\label{Eq:EstR2K}
\Vert r_{2,K,\theta}\Vert_{L^2(\OT)}=\underset{K\to \infty}o\left(\sqrt{\sum_{k=K}^\infty \frac{a_{K,k}^2}{\lambda_{k,\mathcal B}}}\right)
\end{equation} and, thus, that 
\begin{equation}\label{Eq:EstvK}
\Vert v_K-\eta_{0,K,\theta}\Vert_{L^2(\OT)}=\underset{K\to \infty}o\left(\sqrt{\sum_{k=K}^\infty \frac{a_{K,k}^2}{\lambda_{k,\mathcal B}}}\right).\end{equation}
Recall now from Lemma \ref{Le:L2Estimate} that 
\[ \iint_\OT v_K^2\geq c_0\sum_{k=K}^\infty\frac{a_{K,k}^2}{\lambda_{k,\mathcal B}}.\] 
Now let us turn back to the  set $F$, and take any $\phi\in \mathscr C^0([0;T]\times \O)$ such that for any $t$ $\phi(t,\cdot)\in \mathscr C^0_c(F)$. As $\nu_K\,, \nu_\infty\geq 0$, we may take $\phi\geq 0$. Fix a closure point $\nu_\infty$ of the sequence $\left\{\nu_K\right\}_{K\in\N}$. Then
\begin{align*}
\langle \nu_\infty,\phi\rangle&=\lim_{K\to \infty}\langle \nu_K,\phi\rangle
\\&=\lim_{K\to \infty}\frac{\iint_\OT v_K^2 \phi}{\iint_\OT v_K^2}
\\&=\lim_{K\to \infty}\frac{\iint_\OT \eta_{0,K,\theta}^2 \phi+\iint_\OT (v_K-\eta_{0,K,\theta})^2\phi+2\iint_\OT \eta_{0,K,\theta}(v_K-\eta_{0,K,\theta})}{\iint_\OT v_K^2}.
\end{align*}
As $\eta_{0,K,\theta}=\theta w_{0,K}\equiv 0$ on $F$ by the definition of $\theta$, and as $\phi$ is supported in $F$, 
\[ \iint_\OT \eta_{0,K,\theta}^2\phi=0.\] Thus
\begin{align*}
\langle \nu_\infty,\phi\rangle&=\lim_{K\to \infty}\frac{\iint_\OT \eta_{0,K,\theta}^2 \phi+\iint_\OT (v_K-\eta_{0,K,\theta})^2\phi+2\iint_\OT \eta_{0,K,\theta}(v_K-\eta_{0,K,\theta})}{\iint_\OT v_K^2}
\\&=\lim_{K\to \infty}\frac{\iint_\OT (v_K-\eta_{0,K,\theta})^2\phi+2\iint_\OT \eta_{0,K,\theta}(v_K-\eta_{0,K,\theta})}{\iint_\OT v_K^2}
\\&\leq \Vert \phi\Vert_{L^\infty(\OT)}\lim_{K\to \infty}\frac{\Vert v_K-\eta_{0,K,\theta}\Vert_{L^2(\OT)}^2+\Vert \eta_{0,K,\theta}\Vert_{L^2(\O)}\Vert v_K-\eta_{0,K,\theta}\Vert_{L^2(\O)}}{\iint_\OT v_K^2}
\end{align*}
By definition of $\eta_{0,K,\theta}$, 
\[ \Vert \eta_{0,K,\theta}\Vert_{L^2(\OT)}\leq \Vert \theta\Vert_{L^\infty(\O)} \Vert w_{0,K}\Vert_{L^2(\OT)}\leq C \Vert v_K\Vert_{L^2(\OT)}\] for a constant $C$. In the last inequality we used Lemma \ref{Le:L2Estimate}. Combined with Estimate \eqref{Eq:EstvK} this gives
\begin{align*}
\langle \nu_\infty,\phi\rangle&\leq \Vert \phi\Vert_{L^\infty(\OT)}\lim_{K\to \infty}\frac{\Vert v_K-\eta_{0,K,\theta}\Vert_{L^2(\OT)}^2+\Vert \eta_{0,K,\theta}\Vert_{L^2(\O)}\Vert v_K-\eta_{0,K,\theta}\Vert_{L^2(\O)}}{\iint_\OT v_K^2}
\\&=0,
\end{align*}
whence the conclusion.

\end{proof}

\section{Proof of Theorem \ref{Th:Main2}}

The strategy of the proof is to reduce ourselves to the setting of Theorem \ref{Th:LaplaceDimd}. In other words, we need to explain why it is possible to use dirac (in time) type perturbations $h$, so that the equation \eqref{Eq:DotuK} on $\dot u_y$ reduces to a Cauchy problem). To explain why such a construction is possible, we need to give a few words about the optimality conditions for \eqref{Eq:PvIntro2} and the admissible perturbations.

First of all, we argue once again by contradiction and we take $\delta>0$ such that 
\[ \omega^*:=\{-\kappa_0<y^*<\kappa_1\}\cap \{ Z_{y^*}\geq \delta\}\] has positive measure. By inner regularity of the Lebesgue measure we further assume 
\begin{equation}\label{Eq:Closed}
\omega^* \text{ is closed}.
\end{equation}We know that, for any admissible perturbation $h$ at $y^*$ supported in $\omega^*$, we have $\dot J(y^*)[h]=0$. To reach a contradiction, it suffices to construct an admissible perturbation $h$ supported in $\omega^*$ such that  $\ddot J(y^*)[h,h]> 0$.

The main difficulty lies in the structure of the cone of admissible perturbations. This cone, which we denote  $T(y)$ at $y\in \mathcal F$ is defined \cite{CominettiPenot} as follows: $h\in L^1(0,T;L^1(\O))\cap L^\infty(\OT)$ is admissible at $y$ if and only if for any sequence $\{\e_k\}_{k\in \N}$  that converges to 0, there exists a sequence $\{h_k\}_{k\in \N}\in L^1(0,T;L^1(\O))\cap L^\infty(\OT)$ such that $h_k\underset{k\to \infty}\rightarrow h$ in $L^2(0,T;L^2(\O))$ and such that for any $k\in \N$ $y+\e_kh_k\in \mathcal F$. Ideally, we  would choose a perturbation $h$ of the form 
\begin{equation}\label{Eq:Goaly} h=\delta_{t=t_0}h_0(x)\end{equation}  where $t_0\in (0;T)$, for any function $h_0$ with zero mean value and supported in the slice $(\{t=t_0\}\times \O)\cap \omega^*$. Indeed, if the perturbation $h$ is of the form \eqref{Eq:Goaly} the associated $\dot u_{y^*}$ should solve
\begin{equation}\label{Eq:Can}
\begin{cases}
\partial_t \dot u_{y^*}-\Delta \dot u_{y^*}=\partial_uf(t,x,u_{y^*}) \dot u_{y^*}&\text{ in }(t_0;T)\,, 
\\ \mathcal B \dot u_{y^*}=0&\text{ on }(t_0;T)\times \partial \O\,,
\\ \dot u_{y^*}(t=t_0,\cdot)\equiv h_0&\text{ in }\Omega\,, 
\end{cases}\end{equation}
extended by 0 in $(0;t_0)\times \O$ and, provided $\dot u_{y^*}\in L^2(0,T;L^2(\O))$, we would like to say that in that case, if $y^*$ is an optimiser, then
\[ \iint_\OT Z_{y^*}\dot u_{y^*}^2\leq 0.\]In other words, we wish to prove that optimality conditions extend to perturbations $h$ that write as \eqref{Eq:Goaly}, that is, as measures in time. If we can do this, then  we will be able to use Theorem \ref{Th:LaplaceDimd}. Thus we start by proving the following proposition (note that we denote the solutions of equations of the type \eqref{Eq:Can} by $\dot v$, and retain the notation $\dot u$ for the ``standard" notion of Fr\'echet derivative):

\begin{proposition}\label{Pr:ExtensionPerturbations}
For almost every $t_0\in (0;T)$ such that $\omega_{t_0}^*:=(\{t=t_0\}\times \O)\cap \omega$ has positive measure, for any $h_0\in L^2(\O)$ supported in $\omega_{t_0}^*$, extended by zero outside of $\omega_{t_0}^*$ and such that $\int_\O h_0=0$, letting $\dot v$ be the solution of \eqref{Eq:Can} associated with $h_0$, there holds
\[ \iint_\OT Z_{y^*} \dot v^2\leq 0.\]
\end{proposition}

\paragraph{Conclusion of the proof of Theorem \ref{Th:Main2} using Proposition \ref{Pr:ExtensionPerturbations}} Let us show  how the proof of Theorem \ref{Th:Main2}  follows from Proposition \ref{Pr:ExtensionPerturbations}. Fix $t_0$ such that $\mathrm{Vol}(\omega_{t_0}^*)>0$. From the arguments of \cite[Proof of Theorem 1]{MNP2021} we know that, for any $K\in \N$, we may choose $h_K$ supported in $\omega_{t_0}^*$ such that $h_{K}$ writes
\[ h_K=\sum_{k\geq K}a_{K,k}\p_{k,\mathcal B}\,, \sum_{k=K}^\infty a_{K,k}^2=1\,, \int_{\omega_{t_0}^*} h_K=0.\]  Define, for any $K\in \N$, $\dot v_{K,y^*}$ as the solution of \eqref{Eq:Can} associated with $h_K$. From Proposition \ref{Pr:ExtensionPerturbations}, 
\begin{equation}\label{Eq:In} \forall K\in \N\,,\iint_{(0;T)\times \O}Z_{y^*}\dKuu^2+\int_{\O}\partial^2_{uu}j_2|_{u_{y^*}} \dKuu^2=
 \ddot J(y^*)[h_K,h_K]\leq 0.\end{equation} Set now 
 \[ \nu_K:=\frac{\dKuu^2}{\iint_{(0;T)\times \O}\dKuu^2}\] and choose $\nu_\infty$ to be a closure point (in the sense of measures) of $\left\{\nu_K\right\}_{K\in \N}$.  Observe that $\nu_K$ can be considered as a measure in $[t_0;T]\times \overline \O$ as $\dKuu\equiv 0$ in $[0;t_0)$.  As $Z_{y^*}\in \mathscr C^0([t_0;T]\times\overline\O)$ from standard parabolic regularity, dividing \eqref{Eq:In} by $\iint_{(0;T)\times \O}\dKuu^2$ and passing to the limit, we obtain 
 \[\langle \nu_\infty,Z_{y^*}\rangle\leq 0.
 \]  By parabolic regularity, $\partial_u f(t,x,u_{y^*})$ satisfies \eqref{Hyp:q}. We can hence apply Theorem \ref{Th:LaplaceDimd}: $\nu_\infty$ is supported in $\{t=t_0\}\times\omega_{t_0}^*$. As $Z_{y^*}\geq \delta$ on $\omega_{t_0}^*$ and as $\nu_\infty\geq 0$ we have 
 \[ \delta=\delta\langle \nu_\infty,\mathds 1_{[0;T]\times \O}\rangle\leq \langle \nu_\infty,Z_{y^*}\rangle\leq 0,\]
 a contradiction. This concludes the proof.

 Thus, only Proposition \ref{Pr:ExtensionPerturbations} remains to be proved.

\begin{proof}[Proof of Proposition \ref{Pr:ExtensionPerturbations}] 
\textbf{Measure approximation of $\delta_{t=t_0}h_0$}

We know from \cite[Theorem 8.19]{Leoni} that almost every $t\in (0;T)$
is an $L^1(\O)$-Lebesgue point of $\mathds 1_\omega$ in the sense that 
\begin{equation}\label{Eq:Lebesgue}
\lim_{\e \to 0}\fint_{t-\e}^{t+\e}\Vert \mathds 1_\omega(s,\cdot)-\mathds1_\omega(t,\cdot)\Vert_{L^1(\O)}ds=0.\end{equation}
Let $t_0$ be a Lebesgue point such that $\omega_{t_0}^*$ has positive $d$-dimensional measure and let $h_0\in L^2(\O)$. By a standard approximation argument, it suffices to prove the proposition for $h_0\in L^\infty(\O)$. 

Now set, for almost every $s\in (0;T)$, $\omega_s^*:=(\{t=s\}\times \O)\cap \omega^*$ and define, for $\e>0$, 
\begin{equation}\label{Eq:DefHe}h_\e:=\frac1\e\mathds 1_{(t_0-\e;t_0+\e)\times \O}\mathds 1_{\omega^*}(t,x)\left(h_0-\fint_{\omega_t^*}h_0\right)
\end{equation} Clearly $h_\e$ is an admissible perturbation at $y^*$. Furthermore observe that in the sense of measures we have
\begin{equation}\label{Eq:CvMeasure}
h_\e\underset{\e \to 0}\rightharpoonup h_0 
\end{equation} Indeed, for any test function $\Phi\in \mathscr C^0(\OT)$ which me way assume to satisfy $\Vert \Phi\Vert_{L^\infty(\OT)}=1$, we have 
\begin{align*}
\left|\iint_\OT h_\e\Phi-\int_\O h_0\Phi(t_0,\cdot) \right|&\leq \left| \iint_\OT h_\e \Phi(t_0,\cdot)-\int_\O h_0 \Phi(t_0,\cdot) \right|&(=:I_1^\e)
\\&+\left|\iint_\OT h_\e \left(\Phi-\Phi(t_0,\cdot)\right) \right|& (=:I_2^\e)
\end{align*}
By continuity of $\Phi$ and  as $h_\e$ is  (uniformly bounded) Radon measure with support in $(t_0-\e;t_0+\e)\times \O$, we have $I_2^\e\underset{\e \to 0}\rightarrow 0$.
For $I_1^\e$, using $\Vert \Phi\Vert_{L^\infty(\OT)}\leq 1$, we have the estimate
\begin{align*}
I_1^\e&\leq  \fint_{t-\e}^{t+\e} \left|\int_\O h_\e-h_0\right|
\\&\leq \Vert h_0\Vert_{L^\infty(\O)}\fint_{t_0-\e}^{t_0+\e}\Vert \mathds 1_{\omega^*}-\mathds 1_{\omega^*_{t_0}}\Vert_{L^1(\O)}&(=:J_1^\e)
\\&+\fint_{t_0-\e}^{t_0+\e}\left|\fint_{\omega_{t}^*} h_0\right|&(=:J_2^\e)
\end{align*}
$J_1^\e$ converges to 0 as $\e$ converges to zero as $t_0$ was chosen as a Lebesgue point. Furthermore, using the fact that $\fint_{\omega_{t_0}^*}h_0=0$, $J_2^\e$ can be estimated as 
\[ 0\leq J_2^\e\leq \fint_{t_0-\e}^{t_0+\e}\left| \fint_{\omega_{t}^*} h_0-\fint_{\omega_{t_0}^*} h_0\right|\leq \Vert h_0\Vert_{L^\infty(\O)} \fint_{t_0-\e}^{t_0+\e}\Vert \mathds 1_{\omega^*}-\mathds 1_{\omega^*_{t_0}}\Vert_{L^1(\O)}\] which also converges  to $0$ as $t_0$ is a Lebesgue point.

As a consequence, we constructed a sequence of admissible perturbations that converges in the sense of measures to $\delta_{t=t_0}h_0$. To conclude the proof we need to guarantee the convergence of the solutions of \eqref{Eq:DotuK} to $\dot v_y$, the solution of \eqref{Eq:Can}.

\textbf{Convergence of the solutions}

For any $\e>0$ we let $\dot u_\e$ be the solution of \eqref{Eq:DotuK} associated with $h_\e$, and $\dot v_y$ be the solution of \eqref{Eq:Can} associated with $h_0$. Let us show
\begin{equation}\label{Eq:Gallouet}
\Vert \dot u_\e-\dot v_y\Vert_{L^2(0,T;L^2(\O))}\underset{\e \to 0}\rightarrow 0.
\end{equation} 
This suffices to show that 
\[ \iint_\OT Z_{y^*} \dot v_y^2=\lim_{\e \to 0}\iint_\OT Z_{y^*}\dot u_\e^2\leq 0\] and thus provides the conclusion of the proof. \eqref{Eq:Gallouet} follows from two ingredients: one is a general  result of Boccardo \& Gallou\"{e}t \cite[Section IV, Theorem 4]{BoccardoGallouet} that among other things guarantees the well-posedness of the equations at hand, while the second takes advantage of the particular structure of $h_\e$. Let us start with the following theorem:
\begin{theorem*}\label{Th:BGRegu}
\cite[Section IV, Theorem 4]{BoccardoGallouet}
Let $\mathcal M(\OT)$ be the set of Radon measures on $\OT$.
Let $f \in \mathcal M(\OT)$. There exists a unique solution $\theta_f$ to 
\begin{equation}\label{Eq:Theta}
\begin{cases}
\frac{\partial \theta_f}{\partial t}-\Delta \theta_f=f&\text{ in }\OT\,, 
\\ \theta_f(t=0,\cdot)=0\,,
\end{cases}
\end{equation} that further satisfies the following regularity estimates:
\begin{enumerate}
\item $\Vert \theta_f\Vert_{L^\infty(0,T;L^1(\O))}\leq c \Vert f\Vert_{\mathcal M(\OT)}$ for some constant $c=c(\O)$,
\item for any $q\in \left[ 1;\frac{d+2}{d+1}\right)$ there exists a constant $c_q=c_q(\O,T)$ such that 
\[\Vert \theta_f\Vert_{L^q(0,T;W^{1,q}(\O))}\leq c_q\Vert f\Vert_{\mathcal M(\OT)},
\]
\item for any $q\in \left[ 1;\frac{d+2}{d+1}\right)$, the map $f\mapsto \theta_f$ is continuous for the strong $L^q(0,T;W^{1,q}(\O))$ topology on $\theta$.
\end{enumerate}
\end{theorem*}
As we need to apply this regularity result to the solution of an equation with a (bounded) potential we give a more suitable statement, which is just a corollary of Theorem \ref{Th:BGRegu}.
\begin{lemma}\label{Le:Reg}
Let $W\in L^\infty(\OT)$, $f\in \mathcal M(\OT)$.  There exists a unique solution $\eta_f$ to 
\begin{equation}\label{Eq:eta}
\begin{cases}
\frac{\partial \eta_f}{\partial t}-\Delta \eta_f-W\eta_f=f&\text{ in }\OT\,, 
\\ \eta_f(t=0,\cdot)=0\,,
\end{cases}
\end{equation} that further satisfies the following regularity estimates:
\begin{enumerate}
\item $\Vert \eta_f\Vert_{L^\infty(0,T;L^1(\O))}\leq c \Vert f\Vert_{\mathcal M(\OT)}$ for some constant $c=c(\O,\Vert W\Vert_{L^\infty(\OT)})$,
\item for any $q\in \left[ 1;\frac{d+2}{d+1}\right)$ there exists a constant $c_q=c_q(\O,T,\Vert W\Vert_{L^\infty(\OT)})$ such that 
\[\Vert \eta_f\Vert_{L^q(0,T;W^{1,q}(\O))}\leq c_q\Vert f\Vert_{\mathcal M(\OT)},
\]
\item for any $q\in \left[ 1;\frac{d+2}{d+1}\right)$, the map $f\mapsto \eta_f$ is continuous for the strong $L^q(0,T;W^{1,q}(\O))$ topology on $u$.
\end{enumerate}
\end{lemma}

\begin{proof}[Proof of Lemma \ref{Le:Reg}]
We let $\theta_f$ be the solution of \eqref{Eq:Theta} and we let $z$ be the (unique) $L^2(0,T;W^{1,2}_0(\O))$ solution of 
\begin{equation}\label{Eq:z}
\begin{cases}
\frac{\partial z}{\partial t}-\Delta z-Wz=W\theta_f&\text{ in }\OT\,, 
\\ z_f=0&\text{ on }(0,T)\times \partial \O\,, 
\\ z_f(t=0,\cdot)=0&\text{ in }\O.
\end{cases}
\end{equation} Clearly $z+\theta_f$ is a solution of  \eqref{Eq:eta} and the $L^q(0,T;W^{1,q}(\O))$ estimates on $z$ follow from the estimates of Theorem \ref{Th:BGRegu} and from standard elliptic regularity. The conclusion follows.
\end{proof}
Consequently, we can conclude that 
\begin{equation}\label{Eq:BoccardoGall} \dot u_\e\underset{\e \to 0}\rightarrow \dot v_y\text{ in }L^q(0,T;W^{1,q}(\O))\text{ for any $q<\frac{d+2}{d+1}$}.\end{equation}

Let us now exploit the particular structure of $h_\e$. Noticing that 
\begin{align*}
\forall \e>0\,, \Vert h_\e\Vert_{L^2(0,T;L^2(\O))}&\leq \frac{2\Vert h_0\Vert_{L^\infty(\O)}}\e \int_0^T \mathds 1_{(t_0-\e;t_0+\e)}(s)ds
\\&\leq 2\Vert h_0\Vert_{L^\infty(\O)}.
\end{align*}
Consequently, from standard parabolic estimates, 
\[ \sup_{\e\to 0}\Vert u_\e\Vert_{L^2(0,T;W^{1,2}(\O))}\,, \left\Vert\frac{\partial u_\e}{\partial t}\right\Vert_{L^2(0,T;W^{1,2}(\O)}<\infty\] whence the Aubin-Lions lemma entails that $u_\e$ has a strong $L^2(0,T;L^2(\O))$ closure point. From \eqref{Eq:BoccardoGall} this closure point must be $v_y$, which concludes the proof of \eqref{Eq:Gallouet}.

\end{proof}

%

\bibliographystyle{abbrv}
\bibliography{BiblioTwoScale}
\textsc{Idriss MAZARI-FOUQUER}
\\{CEREMADE, UMR CNRS 7534, Universit\'e Paris-Dauphine PSL, \\Place du Mar\'echal De Lattre De Tassigny, 75775 Paris Cedex 16, France, }
\\ \texttt{mazari@ceremade.dauphine.fr}\medskip
\\ \textsc{ Gr\'egoire Nadin}\\ {Laboratoire Jacques-Louis Lions, UMR  CNRS 7598, Sorbonne Universit\'e,\\  Place Jussieu, 75005, Paris, France }
\\ \texttt{gregoire.nadin@sorbonne-universite.fr}

\end{document}